\documentclass[article,11 pt]{amsart}

\RequirePackage{amsmath, amssymb, amsthm, amsfonts}

\newtheorem{theo}{Theorem}[section]

\newtheorem{cor}[theo]{Corollary}
\newtheorem{defin}[theo]{Definition}

\theoremstyle{definition}

\newtheorem{exam}[theo]{Example}

\newtheorem{rem}[theo]{Remark}
\newtheorem{conj}[theo]{Conjecture}
\newtheorem{alg}[theo]{Algorithm}

\newcommand{\Hom}{{\mathrm{Hom}}}
\newcommand{\PSL}{{\mathrm{PSL}}}
\newcommand{\PSU}{{\mathrm{PSU}}}
\newcommand{\SL}{{\mathrm{SL}}}
\newcommand{\PGL}{{\mathrm{PGL}}}
\newcommand{\GL}{{\mathrm{GL}}}

\newcommand{\ord}{{\mathrm{ord}}}
\newcommand{\sgn}{{\mathrm{sgn}}}

\newcommand{\tr}{{\mathrm{tr}}}

\newcommand{\Sz}{{\mathrm{Sz}}}

\newcommand{\Ga}{\Gamma}
\newcommand{\al}{\alpha}
\newcommand{\be}{\beta}
\newcommand{\ga}{\gamma}
\newcommand{\la}{\lambda}

\newcommand{\FQ}{\mathbb{F}_q}

\title{New Beauville surfaces and finite simple groups}
\author{Shelly Garion, Matteo Penegini}
\address{Shelly Garion\\  Max-Planck-Institute for Mathematics\\ D-53111 Bonn, Germany}
\email{shellyg@mpim-bonn.mpg.de}
\address{Matteo Penegini\\
Lehrstuhl Mathematik VIII\\ Universit\"at Bayreuth, NWII\\ D-95440
Bayreuth, Germany} \email{matteo.penegini@uni-bayreuth.de}
\subjclass[2000]{14J10,14J29,20D06,20H10,30F99.}

\begin{document}


\maketitle


\begin{abstract}
In this paper we construct new Beauville surfaces with group either
$\PSL(2,p^e)$, or belonging to some other families of finite simple
groups of Lie type of low Lie rank, or an alternating group, or a
symmetric group, proving a conjecture of Bauer, Catanese and
Grunewald. The proofs rely on probabilistic group theoretical
results of Liebeck and Shalev, on classical results of Macbeath and
on recent results of Marion.
\end{abstract}

\section{Introduction}\label{sect.intro}

A \emph{Beauville surface} $S$ (over $\mathbb{C}$) is a particular
kind of surface isogenous to a higher product of curves, i.e.,
$S=(C_1 \times C_2)/G$ is a quotient of a product of two smooth
curves $C_1$, $C_2$ of genera at least two, modulo a free action
of a finite group $G$, which acts faithfully on each curve. For
Beauville surfaces the quotients $C_i/G$ are isomorphic to
$\mathbb{P}^1$ and both projections $C_i \rightarrow C_i/G \cong
\mathbb{P}^1$ are coverings branched over three points. A
Beauville surface is in particular a minimal surface of general
type.

Beauville surfaces were introduced by F. Catanese in \cite{Cat00},
inspired by a construction of A. Beauville (see \cite{Be}). Catanese was interested in finding examples of \emph{strongly rigid surfaces} $S$, i.e., if $S'$ is another surface homotopically equivalent to $S$ then $S'$ is either biholomorphic or antibiholomorphic to $S$. In \cite{Cat00} he proved that in general if $C_1$ and $C_2$ are two triangle curves with group $G$, if the action of $G$ on the product $C_1 \times C_2$ is free, then $S=(C_1 \times C_2)/G$ is a strongly rigid surface. Since the original example of Beauville had this property he proposed to name these surfaces Beauville surfaces.

A Beauville surface $S$ is either of \emph{mixed} or \emph{unmixed}
type according respectively as the action of $G$ exchanges the two
factors (and then $C_1$ and $C_2$ are isomorphic) or $G$ acts
diagonally on the product $C_1 \times C_2$. The subgroup $G_0$ (of
index $\leq 2$) of $G$ which preserves the ordered pair $(C_1,C_2)$
is then respectively of index $2$ or $1$ in $G$.

Any Beauville surface $S$ can be presented in such a way that the
subgroup $G_0$ of $G$ acts effectively on each of the factors $C_1$
and $C_2$. Catanese called such a presentation \emph{minimal} and
proved its uniqueness in \cite{Cat00}. In this paper we shall consider only minimal Beauville surfaces of unmixed
type so that $G_0 = G$.




Working out the definition of Beauville surfaces one sees that
there is a pure group theoretical condition which characterizes
the groups of Beauville surfaces: the existence of what in \cite{BCG05}
and \cite{BCG06} is called a ''Beauville structure''.

\begin{defin}\label{prop.beau}
An \emph{unmixed Beauville structure} for a finite group $G$ is a
quadruple $(x_1,y_1;x_2,y_2)$  of elements of $G$, which
determines two triples $T_i:=(x_i,y_i,z_i)$ ($i=1,2$) of elements
of $G$ such that :
\begin{enumerate}\renewcommand{\theenumi}{\it \roman{enumi}}
    \item $x_iy_iz_i=1$,
    \item $\langle x_i,y_i \rangle = G$,
    \item $\Sigma(T_1) \cap \Sigma(T_2)= \{1\}$, where
    \[ \Sigma(T_i):= \bigcup_{g \in G} \bigcup^{\infty}_{j=1}
    \{gx^j_ig^{-1},gy^j_ig^{-1},gz^j_ig^{-1}\}.
    \]
\end{enumerate}
\end{defin}
Moreover, $\tau_i:=(\ord(x_i),\ord(y_i),\ord(z_i))$ is called the
\emph{type} of $T_i$, and a type which satisfies the condition:
\[ \frac{1}{\ord(x_i)}+\frac{1}{\ord(y_i)}+\frac{1}{\ord(z_i)}<1
\]
is called \emph{hyperbolic}.


Therefore, the question of which finite groups $G$ admit an
unmixed Beauville structure was raised.
%
%
%
%
%
%
%
This question is deeply related to the question of which finite groups
are quotients of certain triangle groups, which was widely
investigated (see~\cite{Co90, Co10} for a survey). Indeed,
conditions~\textit{(i)} and~\textit{(ii)} of
Definition~\ref{prop.beau} clearly imply that two certain triangle
groups surject onto the finite group $G$. However, the question
about Beauville structures is somewhat more delicate, due to
condition~\textit{(iii)} of Definition~\ref{prop.beau}.

In this paper we show that $\PSL(2,p^e)$, and some other families of finite simple
groups of Lie type of low Lie rank admit a Beauville structure. Moreover for the alternating and
symmetric groups we prove the stronger statement that almost all of these groups admit a Beauville structure with fixed type.
For a detailed account of which other finite groups admit an unmixed Beauville structure we refer to the introduction of \cite{G}.

The main results of this work are the following Theorems.

\begin{theo}\label{thm.unmixed.An}
Let $(r_1, s_1, t_1), (r_2, s_2, t_2)$ be two hyperbolic types.
Then almost all alternating groups $A_n$ admit an unmixed
Beauville structure $(x_1, y_1; x_2, y_2)$ where $(x_1,y_1,
(x_1y_1)^{-1})$ has type $(r_1, s_1, t_1)$ and $(x_2, y_2,
(x_2y_2)^{-1})$ has type $(r_2, s_2, t_2)$.
\end{theo}

A similar theorem also applies for symmetric groups.

\begin{theo}\label{thm.unmixed.Sn}
Let $(r_1, s_1, t_1), (r_2, s_2, t_2)$ be two hyperbolic types,
and assume that at least two of $(r_1,s_1,t_1)$ are even and at
least two of $(r_2,s_2,t_2)$ are even. Then almost all symmetric
groups $S_n$ admit an unmixed Beauville structure $(x_1, y_1; x_2,
y_2)$ where $(x_1,y_1,(x_1y_1)^{-1})$ has type $(r_1, s_1, t_1)$
and $(x_2, y_2, (x_2y_2)^{-1})$ has type $(r_2, s_2, t_2)$.
\end{theo}

This two Theorems completely solve \cite[Conjecture 7.18]{BCG06} by Bauer,
Catanese and Grunewald. The conjecture was inspired by
the proof of Everitt~\cite{Ev} to Higman's Conjecture that every
hyperbolic triangle group surjects to all but finitely many
alternating groups. The proofs of both Theorems are presented in Section~\ref{sect.an},
and are based on results of Liebeck and Shalev \cite{LS04}, who gave
an alternative proof, based on probabilistic group theory, to
Higman's Conjecture. In Section~\ref{sect.an} we also provide theorems similar to Theorem \ref{thm.unmixed.An} and \ref{thm.unmixed.Sn} for surfaces
isogenous to a higher product not necessarily Beauville.

Next, we have the other results on Beauville surfaces.

\begin{theo}\label{thm.unmixed.PSL}
Let $p$ be a prime number, and assume that $q=p^e$ is at least $7$.
Then the group $\PSL(2,q)$ admits an unmixed Beauville structure.
\end{theo}

Its following Corollary is analogous to Theorem~\ref{thm.unmixed.An}
for the family of groups $\{\PSL(2,p)\}_{p \ prime}$.

\begin{cor}\label{cor.PSL.types}
Let $r,s>5$ be two relatively prime integers. Then there are
infinitely many primes $p$ for which the group $\PSL(2,p)$ admits an
unmixed Beauville structure $(x_1, y_1; x_2, y_2)$ where $(x_1,y_1,
(x_1y_1)^{-1})$ has type $(r,r,r)$ and $(x_2, y_2, (x_2y_2)^{-1})$
has type $(s,s,s)$.
\end{cor}

This Theorem and its Corollary are proved in Section~\ref{sect.PSL}.
The proof is based on properties of the groups $\PSL(2,q)$ and on
results of Macbeath~\cite{Ma}.

Moreover, one can generalize Theorem~\ref{thm.unmixed.PSL}, and prove similar
results regarding some other families of finite simple groups of Lie
type of low Lie rank, provided that their defining field is large
enough.

\begin{theo}\label{thm.unmixed.low.rank}
The following finite simple groups of Lie type $G=G(q)$ admit an
unmixed Beauville structure, provided that $q$ is large enough.
\begin{enumerate}
\item Suzuki groups, $G=\Sz(q)=^2\!B_2(q)$, where $q=2^{2e+1}$;
\item Ree groups, $G=^2\!G_2(q)$, where $q=3^{2e+1}$;
\item $G=G_2(q)$, where $q=p^e$ for some prime number $p>3$;
\item $G=^3\!D_4(q)$, where $q=p^e$ for some prime number $p>3$;
\item $G=\PSL(3,q)$, where $q=p^e$ for some prime $p$;
\item $G=\PSU(3,q)$, where $q=p^e$ for some prime $p$.
\end{enumerate}
\end{theo}

This Theorem is proved in Section~\ref{sect.Lie}, and the proof is
based on recent probabilistic group theoretical results of
Marion~\cite{Mar3.09,Mar9.09}, who investigated the possible
surjection of certain triangle groups onto finite simple groups of
Lie type of low Lie rank. The probabilistic group-theoretical approach was further used and generalized
in \cite{GLL}.


Moreover, in the same direction of
Theorems~\ref{thm.unmixed.An} and~\ref{thm.unmixed.Sn}, and inspired
by conjectures of Liebeck and Shalev \cite{LS05} (see also
Section~\ref{sect.conj}), we propose the following Conjecture.

\begin{conj} \label{conj.Lie.gps}
Let $(r_1, s_1, t_1), (r_2, s_2, t_2)$ be two hyperbolic types. If
$G$ is a finite simple classical group of Lie type of Lie rank large
enough, then it admits an unmixed Beauville structure $(x_1, y_1;
x_2, y_2)$, where $(x_1,y_1,(x_1y_1)^{-1})$ has type $(r_1, s_1,
t_1)$ and $(x_2, y_2, (x_2y_2)^{-1})$ has type $(r_2, s_2, t_2)$.
\end{conj}


This paper is organized as follows. In Section~\ref{sect.geom} we
shall present the geometrical background and explain the link
between geometry and group theory. In Section~\ref{sect.ram} one can
find the proofs of the main results.

\begin{rem}
After completing this manuscript, it was brought to our attention
that Fuertes and Jones~\cite{FJ}, have independently and
simultaneously constructed unmixed Beauville structures for the
groups $\PSL(2,q)$, the Suzuki groups $G=\Sz(q)=^2\!B_2(q)$ and the
Ree groups $G=^2\!G_2(q)$, thus proving some of our results
appearing in Theorems~\ref{thm.unmixed.PSL} and
\ref{thm.unmixed.low.rank}. However, their constructions are of
different type.

\end{rem}

\section{Geometrical Background on Ramification
Structures}\label{sect.geom}

We shall denote by $S$ a smooth irreducible complex projective
surface of general type. We shall also use the standard notation in
surface theory, hence we denote by $p_g:=h^0(S,\Omega^2_S)$ the
\emph{geometric genus} of $S$, $q:=h^0(S,\Omega^1_S)$ the
\emph{irregularity} of $S$, $\chi(S)=1+p_g-q$ the \emph{holomorphic
Euler-Poincar\'{e} characteristic}, $e(S)$ the \emph{topological
Euler number}, and $K^2_S$ the \emph{self-intersection of the
canonical divisor} (see e.g.
\cite{BHPV}). In this section, $C$ will always denote a smooth
compact complex curve and $g(C)$ will be its \emph{genus}.

\begin{defin} A surface $S$ is said to be
\emph{isogenous to a higher product of curves} if and only if $S$ is a
quotient $(C_1 \times C_2)/G$, where $C_1$ and $C_2$ are curves of
genus at least two, and $G$ is a finite group acting freely on
$C_1 \times C_2$.
\end{defin}
In \cite{Cat00} it is proven that any surface isogenous to a higher
product has a unique minimal realization as a quotient $(C_1 \times
C_2)/G$, where $G$ is a finite group acting freely and with the
property that no element acts trivially on one of the two factors
$C_i$. From now on we shall work only with minimal realization.

We have two cases: the \emph{mixed} case where the action of $G$
exchanges the two factors (and then $C_1$ and $C_2$ are
isomorphic), and the \emph{unmixed} case where $G$ acts diagonally
on their product.

A surface $S$ isogenous to a higher product is in particular a minimal surface
of general type and it has
\begin{equation*}\label{eq.chie}
K^2_S=8\chi(S), \ 4\chi(S)=e(S), \text{ and }
\end{equation*}
\begin{equation}\label{eq.chi} \chi(S)=\frac{(g(C_1)-1)(g(C_2)-1))}{|G|},
\end{equation}
by Theorem 3.4 of \cite{Cat00}. Moreover, by Serrano
\cite[Proposition 2.2]{Serr},
\begin{equation}\label{eq.irregularity} q(S)=g(C_1/G)+g(C_2/G),
\end{equation}
see also \cite{Cat00} paragraph 3.

A special case of surfaces isogenous to a higher product is given by
Beauville surfaces, which were also defined in \cite{Cat00}.
\begin{defin} A \emph{Beauville surface} is a surface isogenous to a
higher product $S=(C_1 \times C_2)/G$,
which is rigid, i.e., it has no nontrivial deformation.
\end{defin}

\begin{rem}
Every Beauville surface of mixed type has an unramified double
covering which is a Beauville surface of unmixed type.

The rigidity property of the Beauville surfaces is equivalent to the
fact that $C_i/G \cong \mathbb{P}^1$ and that the projections $C_i
\rightarrow C_i/G \cong \mathbb{P}^1$ are branched in three points,
for $i=1,2$. Moreover, by Equation~\eqref{eq.irregularity} one has
$q(S)=0$.
\end{rem}
In the following we shall consider only the \emph{unmixed} case.

From the above Remark one can see that studying Beauville surfaces
(as well as surfaces isogenous to a higher product in general) is
strictly linked to the study of branched covering of complex curves.
We shall recall Riemann's Existence Theorem which translates the
geometric problem of constructing branch covering into a group
theoretical problem. We state it first in great generality.

\begin{defin}\label{defn.orbifold.fuchsian}
Let $g'$ be a non negative integer and $m_1, \dots , m_r$ be
positive integers with $m_i \geq 2$ for all $i$. An \emph{orbifold
surface group} of type $(g' \mid m_1, \dots , m_r)$ is a group
presented as follows:
\begin{multline*}  \Gamma(g' \mid m_1, \dots , m_r):=\langle a_1,b_1, \ldots , a_{g'},b_{g'},c_{1},
\ldots , c_{r} | \\
c^{m_1}_{1}=\dots=c^{m_r}_{r}=\prod_{k=1}^{g'} [a_{k},b_{k}]c_{1}
\cdot \ldots \cdot c_{r} =1 \rangle.
\end{multline*}
If $g'=0$ it is called a \emph{polygonal group}, if $g'=0$ and
$r=3$ it is called a \emph{triangle group}.
\end{defin}
We remark that an orbifold surface group is in particular a
\emph{Fuchsian group} (see e.g. \cite{LS04}).

The following is a reformulation of {\em Riemann's Existence Theorem}:

\begin{theo} A finite group $G$ acts as a group of automorphisms on some
compact Riemann surface $C$ of genus $g$ if and only if there are natural numbers
$g', m_1, \ldots , m_r$, and an orbifold homomorphism
$$\theta \colon \Gamma(g' \mid m_1, \dots , m_r)
\rightarrow G$$  such that
$\ord(\theta(c_i))=m_i \ \forall i$ and such that the Riemann - Hurwitz relation holds:
\begin{equation}\label{eq.RiemHurw}
2g - 2 = |G|\left(2g'-2 + \sum_{i=1}^r \left(1 - \frac{1}{m_i}\right)\right).
\end{equation}

\end{theo}

If this is the case, then $g'$ is the genus of $C':=C/G$. The
$G$-cover $C \rightarrow
C'$ is branched in $r$ points $p_1, \ldots , p_r$ with branching
indices $m_1, \ldots ,
m_r$, respectively.

\medskip\noindent Moreover, if we denote by $x_i \in G$ the image
of $c_i$ under
$\theta$, then
$$
\Sigma(x_1, \ldots , x_r) := \cup_{a \in G} \cup_{i=0}^{\infty}
\{ax_1^ia^{-1}, \ldots ax_r^ia^{-1} \},
$$ is the set of stabilizers for the action of $G$ on $C$.

If we restrict ourselves to the case where the quotient curve is
isomorphic to $\mathbb{P}^1$ then the Theorem suggests the following definition.

\begin{defin}\label{defn.sphergen}
Let $G$ be a finite group and $r \in \mathbb{N}$ with $r \geq 2$.
\begin{itemize}
\item An $r-$tuple $T=(x_1,\ldots,x_r)$ of elements of $G$ is called a
\emph{spherical $r-$system of generators} of $G$ if $ \langle
x_1,\ldots,x_r \rangle=G$ and $x_1 \cdot \ldots \cdot x_r=1$.

\item We say that $T$ is of \emph{type} $\tau:=(m_1, \dots ,m_r)$ if
the orders of $(x_1,\dots,x_r)$ are respectively $(m_1, \dots
,m_r)$.

\item We shall denote:
\[ \mathcal{S}(G,\tau):=\{\textrm{spherical $r-$systems for } G \textrm{ of type } \tau\}.
\]

\item Moreover, two spherical $r_i-$systems $T_1=(x_1, \dots , x_{r_1})$ and $T_2=(x_1, \dots , x_{r_2})$
are said to be \emph{disjoint}, if:
\begin{equation}\label{eq.sigmasetcond} \Sigma(T_1)
\bigcap \Sigma(T_2)= \{ 1 \},
\end{equation}
where
\[ \Sigma(T_i):= \bigcup_{g \in G} \bigcup^{\infty}_{j=0} \bigcup^{r_i}_{k=1} g \cdot x^j_{i,k} \cdot
g^{-1}.
\]
\end{itemize}
\end{defin}

We obtain that the datum of a surface isogenous to a higher product
of unmixed type $S=(C_1 \times C_2)/G$ with $q=0$ is determined,
once we look at the monodromy of each covering of $\mathbb{P}^1$, by
the datum of a finite group $G$ together with two respective
disjoint spherical $r_i-$systems of generators $T_1:=(x_1, \dots ,
x_{r_1})$ and $T_2:=(x_1, \dots , x_{r_2})$, such that the types of
the systems satisfy~\eqref{eq.RiemHurw} with $g'=0$ and respectively
$g=g(C_i)$. The condition of being disjoint ensures that the action
of $G$ on the product of the two curves $C_1 \times C_2$ is free.
Observe that this can be specialized to $r_i=3$, and therefore
can be used to construct Beauville surfaces. This description
suggests the following definition.
\begin{defin} An \emph{unmixed ramification structure} of size $(r_1,r_2)$
for a finite group
$G$, is a pair $(T_1,T_2)$ of tuples $T_1:=(x_1, \dots x_{r_1})$,
$T_2:=(x_1, \dots x_{r_2})$ of elements of $G$,
such that $(T_1,T_2)$ is a disjoint
pair of spherical $r_i-$system of generators of $G$.
\end{defin}

The definition of an \emph{unmixed Beauville structure} given in the
introduction is a special case of the above definition for
$r_1=r_2=3$. Therefore, the problem of finding Beauville surfaces of
unmixed type is now translated into the problem of finding finite
groups $G$ which admit an unmixed Beauville structure.


\section{Ramification Structures for Finite Simple
Groups}\label{sect.ram}


\subsection{Ramification Structures for $A_n$ and
$S_n$}\label{sect.an}

In this section we prove Theorems~\ref{thm.unmixed.An} and
\ref{thm.unmixed.Sn}. The proofs are based on results of Liebeck and
Shalev~\cite{LS04}.

\subsubsection{Theoretical Background -- Higman's Conjecture and a
Theorem of Liebeck and Shalev on Fuchsian groups}

Conder~\cite{Co80} (following Higman) proved that sufficiently large
alternating groups are in fact \emph{Hurwitz groups}, namely they
are quotients of the Hurwitz triangle group $\Delta(2,3,7)$, using
the method of coset diagrams. In fact, Higman had already
conjectured in the late 1960s that every hyperbolic triangle group,
and more generally -- every Fuchsian group, surjects to all but
finitely many alternating groups.

This conjecture was proved by Everitt~\cite{Ev} using the method of
coset diagrams, and later Liebeck and Shalev~\cite{LS04} gave an
alternative proof based on probabilistic group theory. In fact, they
proved a more explicit and general result, which is presented below.

Note that the results of Liebeck and Shalev are applicable to any
Fuchsian group $\Ga$, however, we shall use them only for the case
of orbifold surface groups (see
Definition~\ref{defn.orbifold.fuchsian})
\[
    \Ga=\Ga(g' \mid m_1,\dots,m_r)
\]
that satisfy the inequality
\begin{equation}\label{eq.Fuchsian}
    2g'-2 + \sum_{i=1}^r \bigl( 1-\frac{1}{m_i}\bigr) > 0.
\end{equation}

\begin{defin}
Let $C_i = g_i^{S_n}$ ($1 \leq i \leq r$) be conjugacy classes in
$S_n$, and let $m_i$ be the order of $g_i$. Define $\sgn(C_i):=
\sgn(g_i)$, where $\sgn(g_i)$ is the sign of $g_i$. Moreover
define
\[
    \Hom_{\mathbf{C}}(\Gamma,S_n) = \{\phi\in \Hom(\Gamma,S_n): \phi(c_i) \in C_i
    \text{ for } 1 \leq i \leq r\},
\]
where $\mathbf{C} := (C_1,\dots,C_r)$.
\end{defin}

\begin{defin}
Conjugacy classes in $S_n$ of cycle-shape $(m^k)$, where $n = mk$,
namely, containing $k$ cycles of length $m$ each, are called
\emph{homogeneous}. A conjugacy class having cycle-shape
$(m^k,1^f)$, namely, containing $k$ cycles of length $m$ each and
$f$ fixed points, with $f$ bounded, is called \emph{almost
homogeneous}.
\end{defin}

\begin{theo}~\cite[Theorem 1.9]{LS04}. \label{thm.conj.An}
Let $\Gamma$ be a Fuchsian group, and let $C_i$ ($1\leq i \leq r$)
be conjugacy classes in $S_n$ with cycle-shapes
$(m_i^{k_i},1^{f_i})$, where $f_i<f$ for some constant $f$ and
$\prod_{i=1}^r \sgn(C_i) = 1$. Set $\mathbf{C} = (C_1,\dots,C_r)$.
Then the probability that a random homomorphism in
$\Hom_\mathbf{C}(\Gamma, S_n)$ has image containing $A_n$ tends to
$1$ as $n \rightarrow \infty$.
\end{theo}

Applying this when $\Gamma$ is the triangle group
$\Delta(m_1,m_2,m_3)$, Liebeck and Shalev~\cite{LS04} demonstrate
that three elements, with product $1$, from almost homogeneous
classes $C_1$, $C_2$, $C_3$ of orders $m_1$, $m_2$, $m_3$, randomly
generate $A_n$ or $S_n$, provided $1/m_1 + 1/m_2 + 1/m_3 < 1$. In
particular, when $(m_1,m_2,m_3) = (2, 3, 7)$, one can choose $C_1$,
$C_2$ and $C_3$ as conjugacy classes of even permutations and this
gives random $(2,3,7)$ generation of $A_n$.

Using Theorem~\ref{thm.conj.An}, Liebeck and Shalev deduced the
following Corollary regarding symmetric groups.

\begin{cor}\label{cor.conj.Sn} \cite[Theorem 1.10]{LS04}.
Let $\Gamma=\Gamma(0|m_1,\dots,m_r)$ be a polygonal group which
satisfies the above inequality~\eqref{eq.Fuchsian}, and assume
that at least two of $m_1,\dots,m_r$ are even. Then $\Gamma$
surjects to all but finitely many symmetric groups $S_n$.
\end{cor}


\subsubsection{Beauville Structures and Ramification Structures for $A_n$ and $S_n$}

\begin{proof}[Proof of Theorem~\ref{thm.unmixed.An}]

Assume that $(r_1,s_1,t_1)$ and $(r_2,s_2,t_2)$ are two hyperbolic
types and that $n$ is large enough. By the following
Algorithm~\ref{alg.conj.class}, we choose six almost homogeneous
conjugacy classes in $S_n$, $C_{r_1}$, $C_{s_1}$, $C_{t_1}$,
$C_{r_2}$, $C_{s_2}$, $C_{t_2}$, of orders $r_1,s_1,t_1$,
$r_2,s_2,t_2$ respectively, such that they contain only even
permutations, and they all have different numbers of fixed points.

By Theorem~\ref{thm.conj.An}, the probability that three random
elements $(x_1,y_1,z_1)$ (equivalently $(x_2,y_2,z_2)$) whose
product is $1$, taken from the almost homogeneous conjugacy classes
$(C_{r_1},C_{s_1},C_{t_1})$ (equivalently
$(C_{r_2},C_{s_2},C_{t_2})$) will generate $A_n$, tends to $1$ as $n
\rightarrow \infty$.

This implies that if $n$ is large enough, one can find six elements
$x_1,y_1,z_1$, $x_2,y_2,z_2$ in $A_n$ of orders $r_1,s_1,t_1$,
$r_2,s_2,t_2$ respectively satisfying the following properties.

\begin{itemize}
\item $x_1\in C_{r_1}$, $y_1\in C_{s_1}$, $z_1\in C_{t_1}$, $x_2\in C_{r_2}$, $y_2\in C_{s_2}$, $z_2\in
C_{t_2}$.
\item $x_1y_1z_1=x_2y_2z_2=1$ and $\langle x_1, y_1 \rangle = \langle x_2, y_2 \rangle =
A_n$.
\item For any choice of integers $l_{x_1},l_{y_1},l_{z_1}$,$l_{x_2},l_{y_2},l_{z_2}$, if the
six elements $x_1^{l_{x_1}}$, $y_1^{l_{y_1}}$, $z_1^{l_{z_1}}$,
$x_2^{l_{x_2}}$, $y_2^{l_{y_2}}$, $z_2^{l_{z_2}}$ are not trivial,
then they all belong to different conjugacy classes in $S_n$, since they
all have different numbers of fixed points, and
hence $\Sigma(x_1, y_1, z_1) \bigcap \Sigma(x_2, y_2, z_2) =
\{1_{A_n}\}$.
\end{itemize}

Therefore, if $n$ is large enough, the quadruple $(x_1, y_1; x_2,
y_2)$ is an unmixed Beauville structure for $A_n$, where $(x_1, y_1,
z_1)$ has type $(r_1, s_1, t_1)$ and $(x_2, y_2, z_2)$ has type
$(r_2, s_2, t_2)$.
\end{proof}

\begin{alg}\label{alg.conj.class}
{\em Choosing six almost homogeneous conjugacy classes $C_{r_1}$,
$C_{s_1}$, $C_{t_1}$, $C_{r_2}$, $C_{s_2}$, $C_{t_2}$ in $S_n$, of
orders $r_1,s_1,t_1,r_2,s_2,t_2$ respectively, such that they
contain only even permutations, and they all have different numbers
of fixed points. }

{\bf Step 1:} Sorting $r_1,s_1,t_1,r_2,s_2,t_2$.

Let $m_6\leq \dots \leq m_1$ be the sorted sequence whose elements
are exactly $r_1,s_1,t_1,r_2,s_2,t_2$. Since $n$ can be as large as
we want, we may assume that $n>100m_1$.

{\bf Step 2:} Choosing even integers $k'_i$ ($1 \leq i \leq 6$).

For $1 \leq i \leq 6$, let
\[
    k'_i =
    \begin{cases}
    \lfloor n/m_i\rfloor &\text{ if it is even}, \\
    \lfloor n/m_i\rfloor-1 &\text{ otherwise}.
    \end{cases}
\]

Observe that for $1 \leq i \leq 6$,
\[
     k'_im_i \leq n \leq (k'_i+2)m_i.
\]

{\bf Step 3:} Choosing even integers $k_i$ ($1 \leq i \leq 6$) s.t.
for every $1 \leq i \neq j \leq 6$, $k_im_i \ne k_jm_j$.

Starting from $i=6$ and going down, set $k_i=k'_i$ if $k'_im_i \ne k_jm_j$ for all $j>i$.
It may happen that for some $i < j$, $k'_im_i =k_jm_j $.
In this case, we shall replace $k'_i$ by $k_i$, by choosing it from the set $\{k'_i-2l: 1
\leq l \leq 5\}$
such that for every $j > i$, $k_im_i \ne k_jm_j$. Note
that by our assumption, the integers $k_i$ ($1 \leq i \leq 6$) are
positive.

{\bf Step 4:} Defining the conjugacy classes $C_i$ ($1\leq i \leq
6$).

Assume that $n$ is large enough and let $C_i$ ($1\leq i \leq 6$) be
conjugacy classes in $S_n$ with cycle shapes
\[
    (m_i^{k_i}, 1^{f_i}), \text{ where } f_i = n-k_im_i.
\]

Observe that the conjugacy classes $C_i$ ($1\leq i \leq 6$) satisfy
the following properties:

\begin{enumerate}\renewcommand{\theenumi}{\it \roman{enumi}}
\item For every $1\leq i \leq 6$, $\sgn(C_i)=1$, since $C_i$
contains an even number of cycles (as the $k_i$-s are even).

\item Set $f:=12m_1$. Then for every $1\leq i \leq 6$,
\[
    f_i = n-k_im_i \leq (k'_i+2)m_i - (k'_i-10)m_i = 12m_i \leq 12m_1
    =f,
\]
and hence it is bounded independently of $n$.

\item For every $1\leq i \neq j \leq 6$, $f_i \neq f_j$, since $k_im_i \ne
k_jm_j$.

\item Let $c_i \in C_i$ be some element, then any
non-trivial power $c_i^{l_i}$ has exactly $f_i$ fixed points.

\item By (\emph{iii}) and (\emph{iv}), for any $1\leq i \neq j \leq 6$ and any two
integers $l_i,l_j$, if the powers $c_i^{l_i}$ and $c_j^{l_j}$ are
not trivial, then they belong to different conjugacy classes in
$S_n$.
\end{enumerate}

{\bf Step 5:} Defining the conjugacy classes
$C_{r_1},C_{s_1},C_{t_1},C_{r_2},C_{s_2},C_{t_2}$.

Let $k_{r_1},k_{s_1},k_{t_1},k_{r_2},k_{s_2},k_{t_2}$
(respectively $f_{r_1},f_{s_1},f_{t_1},f_{r_2},f_{s_2},f_{t_2}$)
be the elements of the set $\{k_1,\dots,k_6\}$ (respectively
$\{f_1,\dots,f_6\}$), ordered by the same correspondence between
$\{r_1,s_1,t_1,r_2,s_2,t_2\}$ and $\{m_1,\dots,m_6\}$.

Now, $C_{r_1},C_{s_1},C_{t_1},C_{r_2},C_{s_2},C_{t_2}$ are the six
conjugacy classes in $S_n$ with cycle-shapes
$(r_1^{k_{r_1}},1^{f_{r_1}})$, $(s_1^{k_{s_1}},1^{f_{s_1}})$,
$(t_1^{k_{t_1}},1^{f_{t_1}})$, $(r_2^{k_{r_2}},1^{f_{r_2}})$,
$(s_2^{k_{s_2}},1^{f_{s_2}})$, $(t_2^{k_{t_2}},1^{f_{t_2}})$
respectively.
\end{alg}

In a similar way, we prove Theorem~\ref{thm.unmixed.Sn} regarding
the symmetric groups.

\begin{proof}[Proof of Theorem~\ref{thm.unmixed.Sn}]
Assume that $(r_1,s_1,t_1)$ and $(r_2,s_2,t_2)$ are two hyperbolic
types, such that at least two of $(r_1,s_1,t_1)$ are even and at
least two of $(r_2,s_2,t_2)$ are even, and that $n$ is large enough.
By slightly modifying Algorithm~\ref{alg.conj.class}, we may choose
six almost homogeneous conjugacy classes $C_{r_1}$, $C_{s_1}$,
$C_{t_1}$, $C_{r_2}$, $C_{s_2}$, $C_{t_2}$ in $S_n$, of orders
$r_1,s_1,t_1$, $r_2,s_2,t_2$ respectively, such that two classes of
$C_{r_1},C_{s_1},C_{t_1}$ and two classes of
$C_{r_2},C_{s_2},C_{t_2}$ contain only odd permutations, and all
these classes have different numbers of fixed points.

By Theorem~\ref{thm.conj.An} and Corollary~\ref{cor.conj.Sn}, the
probability that three random elements $(x_1,y_1,z_1)$ (equivalently
$(x_2,y_2,z_2)$) whose product is $1$, taken from the almost
homogeneous conjugacy classes $(C_{r_1},C_{s_1},C_{t_1})$
(equivalently $(C_{r_2},C_{s_2},C_{t_2})$) will generate $S_n$,
tends to $1$ as $n \rightarrow \infty$.

Therefore, if $n$ is large enough, there exists a quadruple $(x_1,
y_1; x_2, y_2)$ which is an unmixed Beauville structure for $S_n$,
where $(x_1, y_1, z_1)$ has type $(r_1, s_1, t_1)$ and $(x_2, y_2,
z_2)$ has type $(r_2, s_2, t_2)$.
\end{proof}

Moreover, since Theorem~\ref{thm.conj.An} and
Corollary~\ref{cor.conj.Sn} apply to any polygonal group, one can
modify Algorithm~\ref{alg.conj.class} and deduce the following
Corollaries.

\begin{cor}\label{cor.unmixed.ram.An}
Let $\tau_1 = (m_{1,1},\dots,m_{1,r_1})$ and $\tau_2 =
(m_{1,1},\dots,m_{1,r_2})$ be two sequences of natural numbers such
that $m_{k,i} \geq 2$ and $\sum_{i=1}^{r_k}(1-1/m_{k,i}) > 2$ for
$k=1,2$. Then almost all alternating groups $A_n$ admit an unmixed
ramification structure of type $(\tau_1,\tau_2)$.
\end{cor}

\begin{cor}\label{cor.unmixed.ram.Sn}
Let $\tau_1 = (m_{1,1},\dots,m_{1,r_1})$ and $\tau_2 =
(m_{1,1},\dots,m_{1,r_2})$ be two sequences of natural numbers such
that $m_{k,i} \geq 2$, at least two of $(m_{k,1},\dots,m_{ k,r_k})$
are even and $\sum_{i=1}^{r_k}(1-1/m_{k,i}) > 2$, for $k=1,2$. Then
almost all symmetric groups $S_n$ admit an unmixed ramification
structure of type $(\tau_1,\tau_2)$.
\end{cor}


\subsection{Beauville Structures for $\PSL(2,p^e)$}\label{sect.PSL}

In this section we prove Theorem~\ref{thm.unmixed.PSL} and
Corollary~\ref{cor.PSL.types}. The proof is based on well-known
properties of $\PSL(2,p^e)$ (see for example \cite{Di,Go,Su}) and on
results of Macbeath~\cite{Ma}.

\subsubsection{Theoretical Background I -- Properties of
$\PSL(2,p^e)$}

Let $q = p^e$, where $p$ is a prime number and $e \geq 1$. Recall
that $\GL(2,q)$ is the group of invertible $2 \times 2$ matrices
over the finite field with $q$ elements, which we denote by $\FQ$,
and $\SL(2,q)$ is the subgroup of $\GL(2,q)$ comprising the matrices
with determinant $1$. Then $\PGL(2,q)$ and $\PSL(2,q)$ are the
quotients of $\GL(2,q)$ and $\SL(2,q)$ by their respective centers.

When $q$ is even, then one can identify $\PSL(2,q)$ with $\SL(2,q)$
and also with $\PGL(2,q)$, and so its order is $q(q-1)(q+1)$. When
$q$ is odd, the orders of $\PGL(2,q)$ and $\PSL(2,q)$ are
$q(q-1)(q+1)$ and $\frac{1}{2}q(q-1)(q+1)$ respectively, and
therefore we can identify $\PSL(2,q)$ with a normal subgroup of
index $2$ in $\PGL(2,q)$. Also recall that $\PSL(2,q)$ is simple for
$q \neq 2,3$.

One can classify the elements of $\PSL(2,q)$ according to the
possible Jordan forms of their pre-images in $\SL(2,q)$. The
following table lists the three types of elements, according to
whether the characteristic polynomial $P(\la):=\la^2 - \al \la +1$
of the matrix $A \in \SL(2,q)$ (where $\al$ is the trace of $A$) has
$0$, $1$ or $2$ distinct roots in $\FQ$.

\begin{center}
\begin{tabular} {|c|c|c|c|c|}
\hline
element & roots       & canonical form in       & order in & conjugacy classes  \\
type    & of $P(\la)$ & $\SL(2,\overline{\mathbb{F}}_p)$ & $\PSL(2,q)$ & \\
\hline \hline

& & & & two conjugacy classes\\
unipotent & $1$ root & $\begin{pmatrix} \pm 1 & 1 \\ 0 & \pm 1 \end{pmatrix}$ & $p$ &  in $\PSL(2,q)$, which \\
 & & $\al=\pm 2$ & & unite in $\PGL(2,q)$ \\
\hline

split & $2$ roots & $\begin{pmatrix} a & 0 \\ 0 & a^{-1}\end{pmatrix}$ & divides $\frac{1}{d}(q-1)$ & for each $\al$: \\
      &           & where $a \in \mathbb{F}_q^*$ & $d=1$ for $q$ even & one conjugacy class \\
      &           & and $a+a^{-1}=\al$ & $d=2$ for $q$ odd & in $\PSL(2,q)$ \\
\hline

non-split & no roots & $\begin{pmatrix} a & 0 \\ 0 & a^q\end{pmatrix}$ & divides $\frac{1}{d}(q+1)$ & for each $\al$: \\
          &          & where $a \in \mathbb{F}_{q^2}^* \setminus \mathbb{F}_q^*$ & $d=1$ for $q$ even & one conjugacy class\\
          &          & $a^{q+1}=1$ & $d=2$ for $q$ odd & in $\PSL(2,q)$ \\
& &  and $a + a^q = \al$ & & \\
\hline
\end{tabular}
\end{center}

\medskip

The subgroups of $\PSL(2,q)$ are well-known (see~\cite{Di,Su}), and
fall into the following three classes.

\medskip

{\bf Class I:} The small triangle subgroups.

These are the \emph{finite} triangle groups $\Delta=\Delta(l,m,n)$,
which can occur if and only if $1/l + 1/m + 1/n > 1$.

This inequality holds only for the following triples:
\begin{itemize}
\item $(2,2,n)$ : $\Delta$ is a dihedral subgroup of order $2n$.
\item $(2,3,3)$ : $\Delta \cong A_4$.
\item $(2,3,4)$ : $\Delta \cong S_4$.
\item $(2,3,5)$ : $\Delta \cong A_5$.
\end{itemize}

Moreover, if at least two of $l,m$ and $n$ are equal to $2$ or if $2
\leq l,m,n \leq 5$, then a subgroup of $\PSL(2,q)$ which is
generated by three elements $x,y$ and $z=(xy)^{-1}$, of orders
$l$,$m$ and $n$ respectively, may be a small triangle group (for a
detailed list of such triples see~\cite[\S 8]{Ma}).

\medskip

{\bf Class II:} Structural subgroups.

Let $\mathcal{B}$ be a subgroup of $\PSL(2,q)$ defined by the images
of the matrices
\[
    \left\{ \begin{pmatrix}
      a & b \\ 0 & a^{-1}
    \end{pmatrix}: a \in \FQ^*,\ b \in \FQ
    \right\},
\]
and let $\mathcal{C}$ be a subgroup of $\PSL(2,\overline {\mathbb{F}}_q)$
defined by the images of the matrices
\[
    \left\{ \begin{pmatrix}
      t & 0 \\ 0 & t^q
    \end{pmatrix}: t \in \mathbb{F}_{q^2}\setminus\FQ,\ t^{q+1}=1
    \right\}.
\]

Any subgroup of $\PSL(2,q)$ which can be conjugated (in
$\PSL(2,\overline {\mathbb{F}}_q$) to a subgroup of either $\mathcal{B}$ or
$\mathcal{C}$ is called a \emph{structural subgroup} of $\PSL(2,q)$.

\medskip

{\bf Class III:} Subfield subgroups.

If $\mathbb{F}_{p^r}$ is a subfield of $\FQ$, then $\PSL(2,p^r)$ is
a subgroup of $\PSL(2,q)$. If the quadratic extension
$\mathbb{F}_{p^{2r}}$ is also a subfield of $\FQ$, then
$\PGL(2,p^r)$ is a subgroup of $\PSL(2,q)$. These groups, as well as
any other subgroup of $\PSL(2,q)$ which is isomorphic to any one of
them, will be referred to as \emph{subfield subgroups} of
$\PSL(2,q)$.



\subsubsection{Theoretical Background II -- Generation Theorems of
Macbeath}

Let $(\al,\be,\ga) \in \FQ^3$, and denote
\[
    E(\al,\be,\ga) := \{ A,B,C \in \SL(2,q): ABC=I,
    \tr A = \al, \tr B = \be, \tr C = \ga \}.
\]

Since all elements in $\PSL(2,q)$ whose pre-images in $\SL(2,q)$
have the same trace are conjugate in $\PGL(2,q)$, all of them have
the same order in $\PSL(2,q)$. Therefore, we may denote by
$\mathcal{O}rd(\al)$ the order in $\PSL(2,q)$ of the image of a
matrix $A\in \SL(2,q)$ whose trace equals $\al$.

\begin{exam}\label{ex.ords}
$\mathcal{O}rd(0)=2$, $\mathcal{O}rd(\pm 1)=3$ and
$\mathcal{O}rd(\pm 2)=p$.
\end{exam}

\begin{theo}\cite[Theorem 1]{Ma}. \label{thm.PSL.traces}
$E(\al,\be,\ga) \neq \emptyset$ for any $(\al,\be,\ga) \in \FQ^3$.
\end{theo}

\begin{defin}
Let $(\al,\be,\ga) \in \FQ^3$. We say that $(\al,\be,\ga)$ is
\emph{singular} if
\[
\al^2 + \be^2 + \ga^2 - \al\be\ga = 4.
\]

Let $l=\mathcal{O}rd(\al)$, $m=\mathcal{O}rd(\be)$ and
$n=\mathcal{O}rd(\ga)$. We say that $(\al,\be,\ga)$ is \emph{small}
if at least two of $l,m,n$ are equal to $2$ or if $2 \leq l,m,n \leq
5$.
\end{defin}

\begin{theo}\cite[Theorem 2]{Ma}.
$(\al,\be,\ga) \in \FQ^3$ is singular if and only if for $(A,B,C)
\in E(\al,\be,\ga)$, the group generated by the images of $A$ and
$B$ is a structural subgroup of $\PSL(2,q)$.
\end{theo}

\begin{theo}\cite[Theorem 4]{Ma}.\label{thm.PSL.gen}
If $(\al,\be,\ga) \in \FQ^3$ is neither singular nor small, then for
any $(A,B,C) \in E(\al,\be,\ga)$, the group generated by the images
of $A$ and $B$ is a subfield subgroup of $\PSL(2,q)$.
\end{theo}

Macbeath~\cite{Ma} used these generation theorems of $\PSL(2,q)$ to
prove that $\PSL(2,q)$ can be generated by two elements one of which
is an involution. Moreover, he classified all the values of $q$ for
which $\PSL(2,q)$ is a \emph{Hurwitz group}, namely a quotient of
the Hurwitz triangle group $\Delta(2,3,7)$.
In addition, he deduced the following.

\begin{cor}\cite[Theorem 7]{Ma}.\label{cor.unip}
If $p$ is an odd prime, then $\PSL(2,p)$ can be generated by two
unipotents whose product is also unipotent.
\end{cor}

\subsubsection{Beauville Structures for $\PSL(2,p^e)$}

\begin{proof}[Proof of~\ref{thm.unmixed.PSL}]

It is known by~\cite[Proposition 3.6]{BCG05} (and can be easily
verified) that $\PSL(2,2) \cong S_3$, $\PSL(2,3) \cong A_4$ and
$\PSL(2,4) \cong \PSL(2,5) \cong A_5$ do not admit an unmixed
Beauville structure.

\medskip

{\bf Case $q=p^e$ odd.}

Let $q \geq 13$ be an odd prime power, then we will construct an
unmixed Beauville structure for $\PSL(2,q)$, $(A_1,B_1;A_2,B_2)$, of
type $(\tau_1,\tau_2)$, where
\begin{align*}
\tau_1 = \left(\frac{q-1}{2},\frac{q-1}{2},\frac{q-1}{2}\right)
\text{  and  } \tau_2 =
\left(\frac{q+1}{2},\frac{q+1}{2},\frac{q+1}{2}\right).
\end{align*}

Let $r=\frac{q-1}{2}$ (respectively $r=\frac{q+1}{2}$), and note
that $r>5$. Let $\al$ be a trace of some diagonal split
(respectively non-split) element $A \in \SL(2,q)$ whose image in
$\PSL(2,q)$ has exact order $r$, and note that $\al \neq 0,\pm 1,
\pm 2$, since $A$ is neither of orders $2$ or $3$ nor unipotent (see
Example~\ref{ex.ords}).

Observe that $(\al,\al,\al)$ is a non-singular triple. Indeed, the
equality $3\al^2 - \al^3 =4$ is equivalent to $(\al-2)^2(\al+1)=0$,
but the latter is not possible.

By Theorem~\ref{thm.PSL.traces}, $E(\al,\al,\al) \neq \emptyset$,
and since $(\al,\al,\al)$ is not singular nor small, for $(A,B,C)
\in E(\al,\al,\al)$, one has $A \neq \pm B$, and moreover, the image
of the subgroup ${\langle A,B \rangle}$ is a subfield subgroup of
$\PSL(2,q)$, by Theorem~\ref{thm.PSL.gen}. However, since the order
of $A$ is exactly $\frac{q-1}{2}$ (respectively $\frac{q+1}{2}$)
then the image of the subgroup ${\langle A,B \rangle}$ is exactly
$\PSL(2,q)$.

Observe that $\frac{q-1}{2}$ and $\frac{q+1}{2}$ are relatively
prime. Hence, if $A_1,A_2 \in \PSL(2,q)$ have orders $\frac{q-1}{2}$
and $\frac{q+1}{2}$ respectively, then every two non-trivial powers
$A_1^i$ and $A_2^j$ have different orders, thus
\[
    \{g_1 A_1^i g_1^{-1}\}_{g_1,i} \cap \{g_2 A_2^j g_2^{-1}\}_{g_2,j} = \{1\},
\]
implying that $\Sigma(A_1,B_1,C_1) \cap \Sigma(A_2,B_2,C_2) =
\{1\}$, as needed.

For smaller values of $q$, a computer calculation (using MAGMA)
shows that $\PSL(2,7)$ admits an unmixed Beauville structure of type
$((4,4,4),(7,7,7))$, $\PSL(2,9) \cong A_6$ admits an unmixed
Beauville structure of type $((4,4,4),(5,5,5))$, and $\PSL(2,11)$
admits an unmixed Beauville structure of type $((5,5,5),(6,6,6))$.

\medskip

{\bf Case $q=2^e$ even.}

Let $q \geq 8$ be an even prime power, then we will construct an
unmixed Beauville structure for $\PSL(2,q)$, $(A_1,B_1;A_2,B_2)$, of
type $(\tau_1,\tau_2)$, where
\begin{align*}
\tau_1 = (q-1,q-1,q-1) \text{  and  } \tau_2 =(q+1,q+1,q+1).
\end{align*}

Let $r=q-1$ (respectively $r=q+1$), and note that $r>5$. Let $\al$
be a trace of some diagonal split (respectively non-split) element
$A \in \PSL(2,q)=\SL(2,q)$ of exact order $r$, and note that $\al
\neq 0,1$, since $A$ is neither unipotent nor of order $3$ (see
Example~\ref{ex.ords}). Then the claim follows as in the previous
case by considering $(A,B,C)\in E(\al,\al,\al)$.



\end{proof}


Observe that the above proof actually shows that the group
$\PSL(2,q)$ can admit many Beauville structures of various types. On
the other hand, if the types are fixed then we can deduce the
following.

\begin{cor}\label{cor.Beau.PSL2p}
Let $p$ be an odd prime, and let $r,s>5$ be two relatively prime
integers each of which divides either $\frac{p-1}{2}$ or
$\frac{p+1}{2}$ or $p$. Then $\PSL(2,p)$ admits an unmixed Beauville
structure of type $((r,r,r),(s,s,s))$.
\end{cor}
\begin{proof}
If each of $r$ and $s$ divides either $\frac{p-1}{2}$ or
$\frac{p+1}{2}$  then the result follows from the proof of
Theorem~\ref{thm.unmixed.PSL}. Otherwise, if $r=p$ (or $s=p$) then
it relies on Corollary~\ref{cor.unip} as well.
\end{proof}

\begin{proof}[proof of Corollary~\ref{cor.PSL.types}]
Without loss of generality we may assume that $s$ is odd. By the
Chinese Remainder Theorem there exists a unique integer $0\leq x
<2rs$ solving the system of simultaneous congruences
\[
    x \equiv 1 \bmod {2r}, \ x \equiv -1 \bmod s.
\]
Note that such $x$ necessarily satisfies that $r \mid \frac{x-1}{2}$
and $s \mid \frac{x+1}{2}$.

Moreover, by Dirichlet's Theorem, the arithmetic progression
$A_{r,s}:=\{2rsn+x: n \in \mathbb{N}\}$ contains infinitely many
primes.

By Corollary~\ref{cor.Beau.PSL2p} the group $\PSL(2,p)$ admits an
unmixed Beauville structure of type $((r,r,r),(s,s,s))$.
\end{proof}

The following two Remarks explain why Theorem~\ref{thm.unmixed.An}
fails to hold in its great generality for the family of groups
$\PSL(2,q)$.

\begin{rem}
Note that unlike the case of Alternating groups, for the group $\PSL(2,p)$,
the condition that $r$ and $s$ are relatively prime is also
\emph{necessary}.

Indeed, assume that $r$ and $s$ are not relatively prime, and let
$d$ be a prime divisor of $\gcd(r,s)$. Then either $d=p$, or $d$
divides $\frac{p-1}{2}$, or $d$ divides $\frac{p+1}{2}$.

Let $A_1$ and $A_2$ be two elements in $\PSL(2,p)$ of orders $r$ and
$s$ respectively, and write $r=r'd$ and $s=s'd$. Assume that $d$
divides $\frac{p-1}{2}$ (or $\frac{p+1}{2}$), then $A_1^{r'}$ and
$A_2^{s'}$ are both of exact order $d$, hence the cyclic subgroups
$\langle A_1^{r'} \rangle$ and $\langle A_2^{s'} \rangle$ are
conjugate in $\PSL(2,p)$, implying that there exist some integers
$k,l$ such that $A_1^{r'k}$ and $A_2^{s'l}$ are conjugate to the
same element of order $d$.

If $d=p$ then $A_1^{r'}$ and $A_2^{s'}$ are both of order $p$, and
so they can be conjugated in $\PSL(2,p)$ to the image of some matrix
{\small $\begin{pmatrix} 1 & a_i \\ 0 & 1
\end{pmatrix}$}, where $a_1,a_2 \in \mathbb{F}_p^*$. Since half the
elements in $S:=\{k:1\leq k \leq p-1\}$ are squares in
$\mathbb{F}_p^*$ and half are non-squares, there exist $k,l \in S$
such that $A_1^{r'k}$ and $A_2^{s'l}$ are both
conjugate in $\PSL(2,p)$ to the image of {\small $\begin{pmatrix} 1 & 1 \\
0 & 1
\end{pmatrix}$}.
\end{rem}

\begin{rem}
Note that Corollary~\ref{cor.Beau.PSL2p} does not hold for the family of groups
$\{\PSL(2,p^e)\}_{p \ prime,\ e \in \mathbb{N}}$,
since one cannot fix a hyperbolic type $(r,s,t)$ and hope that almost all groups
$G=\PSL(2,p^e)$ where $r,s$ and $t$ all divide $|G|$, will be quotients of $\Delta(r,s,t)$.

Indeed, Macbeath~\cite[Theorem 8]{Ma} proved that $\PSL(2,p^e)$ is a
\emph{Hurwitz group}, namely a quotient of $\Delta(2,3,7)$, if
either $e=1$ and $p= 0, \pm 1 \pmod 7$, or $e=3$ and $p= \pm 2, \pm
3 \pmod 7$. Recently, Marion~\cite{Mar09} showed that this
phenomenon occurs in general for any prime hyperbolic type. Namely,
he showed that if $(r,s,t)$ is a hyperbolic triple of primes
and $p$ is a prime number, then there exists a unique integer $e$
such that $\PSL(2,p^e)$ is a quotient of the triangle group
$\Delta(r,s,t)$.

Interestingly, this situation is different for other families of
groups of Lie type of low Lie rank (under the assumption that
$(r,s,t)$ are not too small), as was shown in recent results
of Marion~\cite{Mar3.09,Mar9.09}, which are detailed in
Theorem~\ref{thm.low.rank} below.
\end{rem}


\subsection{Beauville Structures for Other Finite Simple Groups of
Lie Type}\label{sect.Lie}

In this section we prove Theorem~\ref{thm.unmixed.low.rank}
regarding certain families of finite simple groups of Lie type of
low Lie rank. The proof is based on recent results of
Marion~\cite{Mar3.09,Mar9.09}. Moreover, we discuss some Conjectures
on finite simple groups of Lie type in general.


\subsubsection{Beauville Structures for Finite Simple Groups of Low Lie Rank}

\begin{theo}~\cite[Theorems 1,2,4]{Mar3.09} and~\cite[Theorem 1]{Mar9.09}. \label{thm.low.rank}
Let $G$ be one of the finite simple groups of Lie type listed
below, and let $(p_1,p_2,p_3)$ be a hyperbolic triple of primes
$p_1 \leq p_2 \leq p_3$, such that $lcm(p_1,p_2,p_3)$ divides
$|G|$, which, moreover, satisfy the conditions given bellow.
\begin{enumerate}
\item Suzuki groups, $G=^2\!B_2(q)$, where $q=2^{2e+1}$;
\item Ree groups, $G=^2\!G_2(q)$, where $q=3^{2e+1}$;
\item $G=G_2(q)$, where $q=p^e$ for some prime number $p>3$, and
$(p_1,p_2,p_3) \notin \{(2,5,5),(3,3,5),(3,5,5),(5,5,5)\}$;
\item $G=^3\!D_4(q)$, where $q=p^e$ for some prime number $p>3$, and
$(p_1,p_2,p_3)$ are distinct primes, s.t. $\{p_1,p_2\} \neq
\{2,3\}$;
\item $G=\PSL(3,q)$, where $q=p^e$ for some prime $p$, and
$(p_1,p_2,p_3)$ are odd primes;
\item $G=\PSU(3,q)$, where $q=p^e$ for some prime $p$, and
$(p_1,p_2,p_3)$ are odd primes.
\end{enumerate}

Then, if $\phi \in \Hom(\Delta,G)$ is a randomly chosen homomorphism
from the triangle group $\Delta=\Delta(p_1,p_2,p_3)$ to $G$, then
\[
    \lim_{q \rightarrow \infty} Prob\{\phi \text{ is surjective }\}
    =1.
\]
\end{theo}

Now we have all the ingredients needed for the proof of
Theorem~\ref{thm.unmixed.low.rank}.

\begin{proof}[Proof of Theorem~\ref{thm.unmixed.low.rank}]
(1) Let $G=^2\!B_2(q)$, where $q=2^{2e+1}$, then
\[
|G| = q^2(q^2+1)(q-1).
\]
Since $q^2+1 \equiv 0 \pmod 5$, there are at least two prime
numbers, $5$ and some $r>5$, which divide $|G|$. Indeed, $q-1
\equiv 1 \pmod 3$. Moreover $q-1$ is not a power of $5$ since $5
\equiv 1 \pmod 4$, but $q-1 \equiv 3 \pmod 4$. If $q$ is large
enough, then, by Theorem~\ref{thm.low.rank}, the two triangle
groups, $\Delta(5,5,5)$ and $\Delta(r,r,r)$, surject onto $G$, and
hence $G$ admits a Beauville structure of type
$((5,5,5),(r,r,r))$.

(2) Let $G=^2\!G_2(q)$, where $q=3^{2e+1}$, then
\[
|G| = q^3(q-1)(q^3+1).
\]
Since $q^3+1 \equiv 0 \pmod 7$, there are at least two odd prime
numbers, $7$ and some $r$ ($7 \neq r>3$), which divide $|G|$.
Indeed, $q-1$ is not divisible by $3$ nor by $4$. Moreover $q-1$
is not a power of $7$ since $7 \equiv -1 \pmod 8$, but $q-1 \equiv
2 \pmod 8$. If $q$ is large enough, then, by
Theorem~\ref{thm.low.rank}, the two triangle groups,
$\Delta(7,7,7)$ and $\Delta(r,r,r)$, surject onto $G$, and hence
$G$ admits a Beauville structure of type $((7,7,7),(r,r,r))$.

(3) Let $G=G_2(q)$, where $q=p^e$ for some prime number $p>3$, then
\[
    |G|= q^6(q-1)^2(q+1)^2(q^2-q+1)(q^2+q+1),
\]
and so there are at least two distinct prime numbers, $r,s\geq 7$,
which divide $|G|$. To see this, for example, notice that
$q^2+q+1$ and $q^2-q+1$ are odd, coprime, and not divisible by
$5$. So there exists a prime $r \geq 7$ that divides $|G|$. To
find $s$ is enough to prove that $q^2+q+1$ or $q^2-q+1$ are not
powers of $3$. For this is enough to prove that, if they are
divisible by $3$, are not divisible by $9$. If $q^2+q+1$ is
divisible by $3$, then $q \equiv 1 \pmod 3$, and $q^2+q+1 \equiv 3
\pmod 9$. If $q^2-q+1$ is divisible by $3$, then $q \equiv 2 \pmod
3$, and $q^2-q+1 \equiv 3 \pmod  9$. If $q$ is large enough, then,
by Theorem~\ref{thm.low.rank}, the two triangle groups,
$\Delta(r,r,r)$ and $\Delta(s,s,s)$, surject onto $G$, and hence
$G$ admits a Beauville structure of type $((s,s,s),(r,r,r))$.

(4) Let $G=^3\!D_4(q)$, where $q=p^e$ for some prime number $p>3$,
then
\[
    |G| = q^{12}(q-1)^2(q+1)^2(q^2-q+1)^2(q^2+q+1)^2(q^4-q^2+1),
\]
and so there are at least six distinct primes, $p_1=2, p_2=3, p_3,
p_4, p_5, p_6$, which divide $|G|$. Indeed, we can choose for
example $p_3=p$, and $p_4$ and $p_5$ as in (3). For $p_6$ it is
enough to notice that $q^4-q^2+1$ is odd, not divisible by $3$,
and coprime to $q^2+q+1$ and $q^2-q+1$. If $q$ is large enough,
then, by Theorem~\ref{thm.low.rank}, the two triangle groups,
$\Delta(2,p_3,p_5)$ and $\Delta(3,p_4,p_6)$, surject onto $G$, and
hence $G$ admits a Beauville structure of type
$((2,p_3,p_5),(3,p_4,p_6))$.

(5) Let $G=\PSL(3,q)$ (resp. $G=\PSU(3,q)$), where $q=p^e$ for some
prime $p$, then
\[
    |G| = \frac{1}{d}q^3(q-1)^2(q+1)(q^2+q+1),
\]
(resp. $|G| = \frac{1}{d}q^3(q-1)(q+1)^2(q^2-q+1)$ ), where $d=1$ or
$3$.

Hence, there are at least two distinct odd prime numbers, greater
then $3$, $r$ and $s$, which divide $|G|$. Indeed, if $p=2$ and
$e>1$ it is clear, if $p=3$ and $e>1$ then at least one between
$(q-1)$ and $(q+1)$ is not a power of two, hence we can choose
$r$, then proceed as in (3). Else take $r=p$ and proceed as in
(3). If $q$ is large enough, then, by Theorem~\ref{thm.low.rank},
the two triangle groups, $\Delta(r,r,r)$ and $\Delta(s,s,s)$,
surject onto $G$, and hence $G$ admits a Beauville structure of
type $((s,s,s),(r,r,r))$.
\end{proof}


\subsubsection{Conjectures on Finite Simple Classical Groups of Lie Type}
\label{sect.conj}

Liebeck and Shalev raised the following Conjecture in~\cite{LS05}
regarding finite simple classical groups of Lie type.

\begin{conj}[Liebeck-Shalev]
For any Fuchsian group $\Ga$ there is an integer $f(\Ga)$, such
that if $G$ is a finite simple classical group of Lie rank at
least $f(\Ga)$, then the probability that a randomly chosen
homomorphism from $\Ga$ to $G$ is an epimorphism tends to $1$ as
$|G| \to \infty$.
\end{conj}

If this Conjecture holds, it immediately implies that any finite
simple classical group $G$ of Lie rank large enough admits an
unmixed Beauville structure. Indeed, let $s$ and $t$ be two distinct
primes greater than $3$, then the triangle groups $\Delta(s,s,s)$
and $\Delta(t,t,t)$ will surject onto $G$, if $G$ is of Lie rank
large enough, yielding a Beauville structure of type
$((s,s,s),(t,t,t))$ for $G$. Moreover, this Conjecture inspired us
to formulate Conjecture~\ref{conj.Lie.gps}.

\medskip

\textbf{Acknowledgement.} The authors are grateful to Fritz
Grunewald for inspiring and motivating us to work on this problem
together. He is deeply missed.

The authors would like to thank Ingrid Bauer and Fabrizio Catanese
for suggesting the problems, for many useful discussions and for
their helpful suggestions, especially concerning the Theorem of
Liebeck and Shalev.

We are grateful to Bob Guralnick, Martin Liebeck, Alex Lubotzky,
Martin Kassabov, Aner Shalev, Boris Kunyavskii and Eugene Plotkin
for interesting discussions. We would like to thank Claude Marion
for kindly referring us to his recent results.

The authors acknowledge the support of the DFG Forschergruppe 790
''Classification of algebraic surfaces and compact complex
manifolds''. The first author acknowledges the support of the
European Post-Doctoral Institute and the Max-Planck-Institute for
Mathematics in Bonn.

\end{document}